\documentclass[12pt,a4paper]{amsart}
\usepackage{amscd}
\usepackage{amssymb}
\usepackage[centering,text={15.5cm,22cm}]{geometry}
\usepackage{graphicx,color}
\usepackage{tikz}
\usepackage[all]{xy}
\usepackage{mathrsfs}
\usepackage{marvosym}
\usepackage{stmaryrd}
\usepackage{srcltx}
\usepackage{hyperref}
\usepackage{upgreek} %for \upnu

\usepackage{enumitem}

\definecolor{shadecolor}{rgb}{1,0.9,0.7}

\newtheorem{Theorem}{Theorem}[section]
\newtheorem{Lemma}[Theorem]{Lemma}
\newtheorem{Lemma-definition}[Theorem]{Lemma-Definition}
\newtheorem{Proposition}[Theorem]{Proposition}

\newtheorem{Conjecture}[Theorem]{Conjecture}

\theoremstyle{definition}

\newtheorem{definition}[Theorem]{Definition}

\theoremstyle{remark}

\numberwithin{equation}{section}
\numberwithin{figure}{section}

\newcommand {\ol} {\overline}

\newcommand {\lra}  {\longrightarrow}
\newcommand {\im}  {\operatorname{im}}

\def\mydate{\ifcase\month \or January\or February\or March\or
April\or May\or June\or July\or August\or September\or October\or 
November\or December\fi \space\number\day,\space\number\year}

\newcommand{\IP}{\mathbb P}

\newcommand{\F}{\mathbb F}

\newcommand{\mc}{\mathcal}

\DeclareMathOperator{\spec}{Spec}

\DeclareMathOperator{\res}{Res}

\newcommand{\xyR}[1]{
  \xydef@\xymatrixrowsep@{#1}}
\newcommand{\xyC}[1]{
  \xydef@\xymatrixcolsep@{#1}}

\newcommand*{\DashedArrow}[1][]{\mathbin{\tikz [baseline=-0.25ex,-latex, dashed,#1] \draw [#1] (0pt,0.5ex) -- (2.3em,0.5ex);}}

\newcommand*{\MapstoArrow}[1][]{\mathbin{\tikz [baseline=-0.25ex,-latex, |->,#1] \draw [#1] (0pt,0.5ex) -- (2.3em,0.5ex);}}

\theoremstyle{definition}

\theoremstyle{remark}

\numberwithin{equation}{section}

\DeclareMathOperator{\gl1}{GL_1}
\DeclareMathOperator{\pgl1}{PGL_1}
\DeclareMathOperator{\sym}{Sym}
\DeclareMathOperator{\proj}{Proj}
\DeclareMathOperator{\gr}{Gr}

\newcommand{\mb}{\mathbb}
\newcommand{\gla}{\gl1(A)}
\newcommand{\glb}{\gl1(B)}
\newcommand{\pgla}{\pgl1(A)}

\newcommand{\glab}{\gl1(A\otimes_k B)}

\newcommand{\gm}{\mb{G}_{m}}

\newcommand{\ab}{A\otimes_k B}

\begin{document}

%===========================================================

\title[On a conjecture by Voskresenskii]{A constructive approach to a conjecture by Voskresenskii}

\author[M. Florence]{Mathieu Florence}
\address{Institut de Math\'ematiques de Jussieu, Universit\'e Paris 6, place Jussieu 4, 75005 Paris, France}
\email{mathieu.florence@imj-prg.fr}
\thanks{The first  author is partially supported by the French National Agency (Project GeoLie ANR-15-CE40-0012)}

\author[M. van Garrel]{Michel van Garrel}
\address{Fachbereich Mathematik, Universit\"at Hamburg, Bundesstrasse 55, 20146 Hamburg, Germany}
\email{michel.van.garrel@uni-hamburg.de}



\begin{abstract}
Voskresenskii conjectured that stably rational tori are rational. Klyachko proved this assertion for a wide class of tori by general principles. We re-prove Klyachko's result by providing simple explicit birational isomorphisms, and elaborate on some links to torus-based cryptography.
\end{abstract}

\maketitle
\setcounter{tocdepth}{1}
\tableofcontents
\bigskip

%===========================================================
%===========================================================

\section{Introduction}
\label{intro}

Let $k$ be an infinite field of any characteristic.  We denote by $\overline k$ an algebraic closure of $k$. A variety $X$ over $k$ is said to be \emph{rational} if it is birational to a projective space $\IP_{k}^{n}$. A strictly weaker notion is that of stable rationality.
\begin{definition}
Let $X$ be a variety over $k$. $X$ is said to be \emph{stably rational}
if $X\times_{k}\IP_{k}^{m}$ is rational for some $m\geq0$.
\end{definition}
Let $T$ be a linear (=affine) algebraic group over $k$. Then $T$ is said to be an \emph{algebraic torus} if, over an algebraic closure of $k$, it becomes isomorphic to a product of $\mathbb{G}_{m}$'s.
A conjecture of Voskresenskii
(see \cite[p. 68]{Vo}) states that a stably rational torus over
$k$ ought to be rational. This conjecture is widely open.
A result of Klyachko (\cite{Kl}, see also \cite[sec. 6.3]{Vo}) gives a positive answer for a special type of stably rational tori, which we describe now (see section \ref{setup} for a more detailed description).

Let $A$ and $B$ be \'etale $k$-algebras of coprime dimension over
$k$. Denote by $\gl1(A)$ the algebraic group of invertible elements
in $A$. Let $T$ be the quotient of $\gl1(A\otimes_{k}B)$ by
the subgroup generated by $\gl1(A)$ and $\gl1(B)$. Then $T$
is a stably rational $k$-torus and Klyachko shows that it is in fact
rational. However, his proof by general principles does not provide a simple explicit birational isomorphism from $T$ to a projective space.

We remedy to this by re-proving Klyachko's result, constructing a simple birational map from $T$ to a projective space. We expect our construction to generalize to the situation where $A$ and $B$ are \emph{any } not necessarily commutative finite-dimensional $k$-algebras, of coprime dimension over $k$ (in that case $T$ is not necessarily a torus, or even an algebraic group).

In section \ref{application}, we explore applications of our explicit birational maps to torus-based cryptography. Following the methods developed by Rubin-Silverberg in \cite{RS}, we propose more general compression algorithms.

\section*{Acknowledgement} A large part of this work was accomplished while the first author was visiting the Korea Institute for Advanced Study (KIAS), where the second author at the time was a research fellow. The authors thank KIAS for the hospitality and excellent research environment. The comments of the referee greatly enhanced the quality of the paper. The authors are particularly thankful for the connection made with cryptography.

\section{Setup and statement of results}
\label{setup}

Let $k$ be an infinite field of any characteristic and let $V$ be a finite-dimensional $k$-vector space. We start by recalling some $k$-schemes that are associated to $V$. The \emph{affine space of $V$}, denoted by $\mb{A}(V)$, is defined as the functor
\[
X \mapsto \mb{A}(V)(X) := V\otimes_k \Gamma(X,\mc{O}_X),
\]
from $k$-schemes to sets.
It is represented by the affine scheme $\spec \left(\sym V^*\right)$.\\
The \emph{projective space of $V$}, denoted $\mb{P}(V)$, represents the (functor of)  locally free  submodules of rank one $N \subset V$,  such that the quotient $V/N$ is locally free. It is defined to be
\[
\mb{P}(V):=\left(\mb{A}(V)-\left\{0\right\}\right)/\gm = \proj \left(\sym V^*\right).
\]
Let $A$ be a not necessarily commutative (unital) $k$-algebra of finite dimension.  The linear algebraic group $\gla$ is defined as the functor
\[
X \mapsto \gla(X) := \left( A\otimes_k \Gamma(X,\mc{O}_X)\right)^\times,
\]
from $k$-schemes to groups. It is represented by the closed subscheme of $\mathbb{A}(A\oplus A)$ given by the equation $xy=1$.  One has a canonical injective homomorphism of algebraic groups $$ 1 \longrightarrow \gm \longrightarrow \gla,$$
and can form the quotient
\[
 \pgla := \gla/\gm,
\]
which is a linear algebraic group.\\ For the remainder, assume that $A$ is commutative.\\
Then, $\gla$ is canonically isomorphic to the Weil restriction of scalars $\res_{A/k}(\gm)$.   Let $M$ be an $A$-module, which is locally free of finite rank. The projective space  $\mb{P}(M)$ is defined over $\spec(A)$. In this work, it shall be viewed as a $k$-variety,  by Weil scalar restriction. More explicitly, we set
\[
\mb{P}_A(M) := \res_{A/k}(\mb{P}(M)).
\]
Consider two finite-dimensional commutative $k$-algebras $A$ and $B$. We have an exact sequence
\[
\xymatrix{
1 \ar[r] & \gm \ar[r]^(0.27)i & \gla\times\glb \ar[r]^(0.55)j & \glab,
}
\]
\[
i(x)=(x,x^{-1}), \; j(a,b)= a\otimes b.
\]
Put $H(A,B)=\im(j)$. We will consider the quotient
\begin{equation}\label{qab}
Q(A,B):=\glab/H(A,B).
\end{equation}
It follows from \cite[section 6.1, Theorem 1]{Vo} that $Q(A,B)$ is stably rational.\\
Recall that a $k$-algebra $A$ is said to be \'etale if one of the two following equivalent conditions holds:
\begin{itemize}
\item $A\cong \prod_{i=1}^n k_i$, where the $k_i$ are finite separable field extensions of $k$.
\item $A\otimes_k \overline{k} $, as a $ \overline{k} $-algebra, is isomorphic to a finite product of copies of $ \overline{k}. $
\end{itemize}
The main result of this paper is to re-prove, in a constructive fashion, the following result.

\begin{Theorem}[Klyachko in \cite{Kl}, see also \cite{Vo}, section 6.3]\label{mainthm}
Let $k$ be an infinite field of any characteristic and let $A$ and $B$ be two \'etale $k$-algebras of finite dimension. Assume  that $\dim(A)$ and $\dim(B)$ are coprime.
Then $Q(A,B)$ is $k$-rational.
\end{Theorem}

Note that the proof of Theorem \ref{mainthm} that we provide in section \ref{rationality} is via explicit birational isomorphisms, whereas Klyachko's original proof is by general principles.\\
Recall that an algebraic $k$-torus of dimension $d$ is a $k$-group scheme $T$ such that
$$
T\times_{k} \overline{k} \cong \mb{G}^{\scriptstyle d}_{{\scriptscriptstyle m,\overline{k}}}.
$$
The following conjecture states that for algebraic tori, stable rationality is equivalent to rationality.

\begin{Conjecture}[Voskresenskii, see \cite{Vo}, section 6.2]\label{conj}
Stably rational $k$-tori are $k$-rational. 
\end{Conjecture}

Theorem \ref{mainthm} thus provides a positive proof of Conjecture \ref{conj}, in a particular case.

\section{Proof of the Theorem}
\label{rationality}

Let $A$ and $B$ be \'etale $k$-algebras of coprime dimensions (over $k$) $a$ and $b$, respectively. Being invertible is an open condition, so that $\mathrm{GL} _1(A\otimes B)$ is a nonempty open subvariety of $\mb A (A \otimes B) .$ 
Choose integers $0 < u \leq b$ and $0 < v \leq a$ such that
\[
ua+vb=ab+1.
\]
This is possible since $a$ and $b$ are chosen to be coprime to each other.
For a $k$-vector subspace $W\subset A\otimes_k B$,  containing $1$, denote by $$\IP_1(W)  \subset \mb P(W)$$ the non-empty open subvariety consisting of lines directed by an invertible element of  $W$.

\begin{Proposition}\label{claim} There exist  $k$-vector subspaces $U\in \gr(u,B)(k)$ and $ V\in \gr(v,A)(k)$, both containing $1$, such that the morphism below is a birational isomorphism: 
\begin{equation}\label{phi 1}
\begin{split}
\phi_1 \; : \; \IP_1(V\otimes_k B) \times \IP_1(A\otimes_k U) \quad & \DashedArrow[->,densely dashed    ] \quad \IP_1(A\otimes_k B) = \pgl1(A\otimes_k B), \\ 
\left( x, y \right) \quad & \MapstoArrow \quad xy^{-1}.
\end{split}
\end{equation}
\end{Proposition}

\begin{proof}[Proof in the case of fields] We first prove the assertion in the case that $A$ and $B$ are fields. Then $A\otimes_k B$ is a field as well, because $a$ and $b$ are  coprime.  Take arbitrary $U$ and $V$ as in the statement. We claim that $\phi_1$ then is a birational isomorphism. Consider the fibers of $\phi_1$. An invertible $k$-rational point of $\pgl1(A\otimes_k B)$ is given by the class of $t\in (A\otimes_k B)^\times$. The fiber over that class consists of (the projectivization of)
\[
\left\{(x,y) \in (V\otimes_k B) \oplus (A\otimes_k U) \; | \; x=yt \right\},
\]
where 
$(V\otimes_k B) \oplus (A\otimes_k U)$ is a vector $k$-space of dimension $vb+au=ab+1$. Hence the equation $x=yt$ in $A\otimes_k B$ breaks down into a homogeneous linear system of $ab$ equations in $ab+1$ variables. It follows that it has a non-trivial solution $(x,y)$ over $k$. Since $A\otimes_k B$ is a field, both $x$ and $y$ are invertible. This shows that the fiber of $\phi_1$ at $t$ is non-empty, even isomorphic to a non-empty open of a projective space. But one may base-change from $k$ to the function field $K$ of $ \pgl1(A\otimes_k B)$, and reproduce the previous arguments with $K$ instead of $k$ (note that $K/k$ is purely transcendental, hence $A\otimes_k K$ and $B\otimes_k K$ are still fields). We thus get that the generic fiber   of $\phi_1$ is $K$-rational.  But the source and target of $\phi_1$ have the same dimension $ab-1$. Hence, as asserted, $\phi_1$ is a birational isomorphism. 
\end{proof}

\begin{proof}[Proof in the general case.]
It is a specialization argument as follows. We start by introducing the polynomial algebra (in $a+b$ variables)
\[
\mathcal K:=k[x_0,\ldots, x_{a-1}, y_0, \ldots y_{b-1}],
\]
and denote by $\tilde {\mathcal K}$ its field of fractions.
Set
\[
\mathcal A:=K[T]/<T^{a}+x_{a-1}T^{a-1}+\ldots+x_1 T +x_0>
\]
and
\[
\mathcal B:=K[T]/<T^{b}+y_{b-1}T^{b-1}+\ldots+y_1 T +y_0>,
\]
and put
\[
\tilde{ \mathcal A}:=\mathcal A \otimes_{\mathcal K} \tilde{ \mathcal K}
\]
as well as
\[
\tilde{ \mathcal B}:=\mathcal B \otimes_{\mathcal K} \tilde{ \mathcal K}.
\]
Then $\tilde{ \mathcal A}$ (resp. $\tilde{ \mathcal B}$) is an \'etale $\tilde {\mathcal K}$-algebra of degree $a$ (resp. $b$). It is clearly a field. Pick $\tilde {\mathcal K}$-subspaces  $\tilde {\mathcal U}\in \gr(u,\tilde {\mathcal B})(\tilde {\mathcal K})$ and $\tilde {\mathcal V}\in \gr(v,\tilde {\mathcal A})(\tilde {\mathcal K})$, both containing $1$. By what precedes, the $\tilde {\mathcal K}$-morphism \begin{equation*}
\begin{split}
\tilde \Phi_1 \; : \; \IP_1(\tilde {\mathcal V} \otimes_{\tilde {\mathcal K}} \tilde {\mathcal B}) \times  \IP_1(\tilde {\mathcal A} \otimes_{\tilde {\mathcal K}} \tilde {\mathcal U}) \quad & \DashedArrow[->,densely dashed    ] \quad \IP_1(\tilde {\mathcal A} \otimes_{\tilde {\mathcal K}} \tilde {\mathcal B}) , \\ 
\left( x, y \right) \quad & \MapstoArrow \quad xy^{-1}
\end{split}
\end{equation*}
is a birational isomorphism. Since all above schemes are of finite presentation over $\tilde {\mathcal K}$, they, as well as $\Phi_1$, are actually defined over a nonempty open subscheme of $\spec(\mathcal K)$. More precisely, there exists a nonzero element $F \in \mathcal K$, such that, denoting by $\mathcal K_{(F)}$ the $k$-algebra obtained by inverting $F$ in $\mathcal K$, the following holds:
\begin{enumerate}[label=(\alph*)]
\item The $\mathcal K_{(F)}$-algebras   $\mathcal A_{(F)}$   and $\mathcal B_{(F)}$ are \'etale.
\item The subspaces  $\tilde {\mathcal U}$ and $\tilde {\mathcal V}$ are defined over $\mathcal K_{(F)}$, i.e., are given by elements  $ {\mathcal U}\in \gr(u, {\mathcal B})( {\mathcal K_{(F)}})$  and $ {\mathcal V}\in \gr(u, {\mathcal A})( {\mathcal K_{(F)}})$, respectively.
\item The  $ {\mathcal K_{(F)}}$-morphism \begin{equation*}
\begin{split}
 \Phi_1 \; : \; \IP_1( {\mathcal V} \otimes_{ {\mathcal K_{(F)}}}  {\mathcal B_{(F)}}) \times  \IP_1( {\mathcal A_{(F)}} \otimes_{ {\mathcal K_{(F)}}}  {\mathcal U}) \quad & \DashedArrow[->,densely dashed    ] \quad \IP_1( {\mathcal A_{(F)}} \otimes_{ {\mathcal K_{(F)}}}  {\mathcal B_{(F)}} ) , \\ 
\left( x, y \right) \quad & \MapstoArrow \quad xy^{-1}
\end{split}
\end{equation*} is a birational isomorphism.
\end{enumerate}
But the  \'etale $\tilde {\mathcal K}$-algebras $\tilde {\mathcal A}$ and  $\tilde {\mathcal B}$ are \textit{versal}, in the sense of \cite[Definition 5.1, see also section 24.6]{Coh}. Hence, there exists a $k$-morphism
\[
\theta:   {\mathcal K_{(F)}} \lra k
\]
such that $ {\mathcal A_{(F)}} \otimes_\theta k $ is isomorphic to $A$ (resp.   such that $ {\mathcal B_{(F)}} \otimes_\theta k $ is isomorphic to $B).$   Put  $V:=\mathcal V  \otimes_\theta k$ and $U:=\mathcal U  \otimes_\theta k$. Then $U$ (resp. $V$) belongs to $ \gr(u,B)(k)$ (resp. to $ \gr(v,A)(k)$), and the specialization of $\Phi_1$ via $\theta$ yields the birational isomorphism $\phi_1$.
This finishes the proof of Proposition \ref{claim}.
\end{proof}

Note that $\IP_1(V\otimes_k B) \times \IP_1(A\otimes_k U)$ is open in $\IP(V\otimes_k B) \times \IP(A\otimes_k U)$, and that $\IP_1(A\otimes_k B)$ is open in $\IP(A\otimes_k B)$. Hence the map $\phi_1$ of \eqref{phi 1} extends to a birational isomorphism
\begin{equation}\label{phi}
\phi \; : \; \IP(V\otimes_k B) \times \IP(A\otimes_k U) \quad  \DashedArrow[->,densely dashed    ] \quad \IP(A\otimes_k B).
\end{equation}
Generically, $G:=\gla / \gm \times \glb / \gm$ acts freely on both sides of \eqref{phi}. %We quotient out both sides of \eqref{phi} by the action of $\gla \times \glb$, and use 
We have the identifications as birational quotients:
\begin{align*}
\IP(V\otimes_k B) / \left(\glb/\gm\right) & \equiv  \left(V\otimes_k B / \gm \right) / \left( \glb/\gm \right) \\ & \equiv  \left(V\otimes_k B \right) / \glb \equiv  \IP_B(V\otimes_k B).
\end{align*}
Since the action of $\glb/\gm$ on $V\otimes_k B / \gm$ is generically free, we conclude that $\dim \IP_B(V\otimes_k B) = vb - b$. Similarly,
$$
\IP(A\otimes_k U) / \gla \equiv  \IP_A(A\otimes_k U)
$$
is of dimension $au-a$.\\
On the right hand side of the map of \eqref{phi}, we take the following birational quotient:
%\begin{align*}
\[
\IP(\ab)  / G     % \, / \, \left({\scriptstyle \gla/\gm \times \glb/\gm}\right) \; 
% \equiv \left(\ab/\gm\right)  / G  % \, / \, {\scriptstyle \gla \times \glb} \\
 \equiv  \left(\ab/\gm\right)  / G .  %\, / \, {\scriptstyle \left(\gla / \gm\right) \times \left(\glb / \gm\right)}
 \]
%\end{align*}
As $G$ acts generically freely, the dimension of this quotient is $ab-a-b+1$.
For an $A$-module $M$, recall from section \ref{setup} that we defined $\mb{P}_A(M)$ to be the Weil scalar restriction $\res_{A/k}(\mb{P}(M))$.

\begin{Lemma}\label{lem quot}
The map $\phi$ of \eqref{phi} induces a birational isomorphism
\begin{equation*}%\label{phi ind}
\ol{\phi} \; : \; \IP_B(V\otimes_k B) \times \IP_A(A\otimes_k U) \quad  \DashedArrow[->,densely dashed    ] \quad \IP(\ab) \, / \, {\scriptstyle \gla \times \glb}.
\end{equation*}
\end{Lemma}

\begin{proof}
The dimensions of both quotients agree. Since the map is a birational isomorphism before taking the quotient, we only need to show that it descends to the quotient. But that is clear since the map is given by taking the inverse and multiplication.
\end{proof}

Finally, note that $ \IP(\ab) \, / \, {\scriptstyle \gla \times \glb}$ is birational to $Q(A,B)$. This then completes the proof of Theorem \ref{mainthm}, as both $\IP_B(V\otimes_k B)$ and $\IP_A(A\otimes_k U)$ are rational.

\section{An application to cryptography}
\label{application}

Our explicit birational maps open up some new venues for torus-based cryptography. Using finite cyclic groups for public key encryption is an old idea, cf. \cite[chapter 8]{MOV}. Rubin-Silverberg in \cite{RS} suggested using rational algebraic tori defined over finite fields. The advantage is in term of computational gain. Representing most elements of the torus as elements of an affine space over a finite field yields efficiency gains in the transmitted information. Let $q$ be a prime power and choose $n\geq1$ to be a square-free integer. If $\F_q \subseteq L \subsetneq \F_{q^n}$ is an intermediate field, recall that there is a norm map
\[
N_{\: \F_{q^n}/L} : \res_{\: \F_{q^n}/ \: \F_q}\mathbb{G}_m  \to  \res_{\: L/ \: \F_q}\mathbb{G}_m.
\]
Following \cite{RS}, consider
\[
T_n := \bigcap_{\F_q \subseteq L \subsetneq \F_{q^n}} \ker\left( N_{\: \F_{q^n}/L} \right) \text{ and } G_{q,n} := T_n(\F_q).
\]
For encryption purposes, $G_{q,n}$ is the cryptographically most significant part of $\F_{q^n}^\times$ and $G_{q,n}$, albeit smaller, inherits the security of $\F_{q^n}^\times$. See \cite{RS} for more details.
$G_{q,n}$ is a torus over $\F_q$ of dimension $\phi(n)$, where $\phi$ denotes Euler's phi function. Assuming that it is rational, one then would like to (computationally) compress elements of $G_{q,n}$ via a \emph{compression map} (birational map)
\begin{equation*}
f \; : \; G_{q,n} \quad  \DashedArrow[->,densely dashed    ] \quad \mathbb{F}_q^{\varphi(n)}\end{equation*}
that has an efficiently computable inverse $j$. Since $G_{q,n}<\F_{q^n}^\times$, the latter being of dimension $n$ over $\F_q$, sending $f(x)$ instead of $x\in G_{q,n}$ yields an efficiency gain (in bits) of $n/\phi(n)$. Based on this idea, Rubin-Silverberg introduce two compression algorithms inducing efficient public key cryptosystems that they name $\mathbb{T}_2$ and CEILIDH. They also explain how to extend their algorithms to all $G_{q,n}$, provided that a compression map is known. Note that the encryption is restricted to the open part of $G_{q,n}$ where $f$ and $j$ are mutually inverse. This part is large if $q$ is large, see the discussion in \cite{RS}.  Moreover, they limit their discussion to when $n$ is the product of up to two distinct primes. In particular, they consider $n=2$ for  $\mathbb{T}_2$ and $n=6$ for CEILIDH (to yield secure encryption, $q$ should be large). If $n$ is the product of at least three primes, it is not known, though conjectured by Voskresenskii, that $G_{q,n}$ is rational.

 $\mathbb{T}_2$ and CEILIDH are based on explicit birational compression maps that Rubin-Silverberg construct from Galois extensions. They rely on choosing generators for these extensions. Our setting extends the groups beyond $G_{q,n}$ and does not rely on the extension being Galois, nor on choosing generators.

For the remainder, let $A=\F_{q^a}$ and $B=\F_{q^b}$, where $a$ and $b$ are coprime. Our birational decompression map
\begin{equation*}%\label{decompmap}
\phi(A,B) \; : \; \IP_B(V\otimes_{\F_q} B) \times \IP_A(A\otimes_{\F_q} U) \quad  \DashedArrow[->,densely dashed    ] \quad Q(A,B)
\end{equation*}
solely depends on the choice of the $\F_q$-vector subspaces $U\in \gr(u,B)(\F_q)$ and $ V\in \gr(v,A)(\F_q)$ of Proposition \ref{claim}. Note that, though Theorem \ref{mainthm} is a priori stated for infinite fields, it is easy to see that it actually holds for $k$ finite, when $A$ and $B$ are fields. 
Furthermore, since $a$ and $b$ are coprime to each other, $A\otimes_{\F_q} B=\F_{q^{ab}}$ and
\[
Q(A,B)(\F_q) = \F_{q^{ab}}^\times/ \langle \F_{q^a}^\times,\F_{q^b}^\times\rangle,
\]
where, cf. \eqref{qab}, $\langle \F_{q^a}^\times,\F_{q^b}^\times\rangle = H(\F_{q^a},\F_{q^b})(\F_q)$ is the subgroup generated by $\F_{q^a}^\times$ and $\F_{q^b}^\times$. If in addition $a$ and $b$ are distinct primes (or $b=1$ and $a$ is prime), then
\[
Q(A,B)(\F_q) \cong    Q_{q,ab},
\]
which is the case developed in \cite{RS}. Note that our compression maps differ, and work for all choices of primes $a$ and $b$.

In order to have a computationally efficient extension of Rubin-Silverberg's algorithms to $Q(A,B)$, two conditions must be satisfied. Firstly, the ratio $ab/\phi(ab)$ should be large. Second and most crucially, the $\F_q$-vector subspaces $U$ and $V$ should be chosen such that the birational inverse of $\phi(A,B)$ is computed fast. As explained in the proof of Proposition \ref{claim}, calculating this inverse is obtained through solving linear equations. Suitable choices of $U$ and $V$ will lead to computationally efficient algorithms. We leave the specifics of implementation to future considerations.

%===========================================================


\begin{thebibliography}{77}	



\bibitem{Coh}
S. Garibaldi, A. Merkurjev and J.-P. Serre,
\emph{Cohomological Invariants in Galois Cohomology},
University Lecture Series,
Vol. 28,
AMS, 2003.




\bibitem{Kl}
A. A. Klyachko,
\emph{On rationality of tori with a cyclic splitting field},
Arithmetic and Geometry of Varieties,
Kuibyshev Univ. Press,
Kuibyshev, 1988, pp. 73-78 (Russian).


\bibitem{MOV}
A. J. Menezes, P. C. van Oorschot and S. A. Vanstone,
\emph{Handbook of applied cryptography},
CRC Press, Boca Raton, FL, 1997.


\bibitem{RS}
K. Rubin and A. Silverberg,
\emph{Compression in finite fields and torus-based cryptography},
SIAM J. Comput.,
Vol. 37, No. 5 (2008),
pp. 1401-1428.


\bibitem{Vo}
V. E. Voskresenskii,
\emph{Algebraic groups and their birational invariants},
Translations of mathematical monographs,
Vol. 179,
AMS, 1998.



\end{thebibliography}
\end{document}